\documentclass[12pt]{amsart}
\usepackage[left=3cm,right=3cm,top=4cm,bottom=4cm]{geometry}
\usepackage{amssymb}
\usepackage{graphicx}
\usepackage[cmtip,all]{xy}
\usepackage{mathtools}
\usepackage{mathrsfs}
\usepackage{bm}
\usepackage{blkarray, bigstrut}
\usepackage{booktabs}
\usepackage{wasysym}
\usepackage{tikz,tikz-cd}
\usepackage[colorlinks=true]{hyperref}
\usepackage{scalerel}
\usepackage{stackengine,wasysym}

\newcommand\blfootnote[1]{%
  \begingroup
  \renewcommand\thefootnote{}\footnote{#1}%
  \addtocounter{footnote}{-1}%
  \endgroup
}
\newcommand{\proj}{{\text{proj}}}
\newtheorem{theorem}{Theorem}[section]
\newtheorem{lemma}[theorem]{Lemma}
\newtheorem{corollary}[theorem]{Corollary}

\newtheorem{proposition}[theorem]{Proposition}

\theoremstyle{definition}
\newtheorem{definition}[theorem]{Definition}
\newtheorem{example}[theorem]{Example}
\theoremstyle{remark}
\newtheorem{remark}[theorem]{Remark}
\theoremstyle{remark}
{

}
\numberwithin{equation}{section}

\newcommand{\CC}{{\mathbb C}}

\newcommand\reallywidetilde[1]{\ThisStyle{%
  \setbox0=\hbox{$\SavedStyle#1$}%
  \stackengine{-.1\LMpt}{$\SavedStyle#1$}{%
    \stretchto{\scaleto{\SavedStyle\mkern.2mu\AC}{.5150\wd0}}{.6\ht0}%
  }{O}{c}{F}{T}{S}%
}}

\def\test#1{$%
  \reallywidetilde{#1}\,
$\par}
\newcommand{\bigslant}[2]{{\raisebox{.2em}{$#1$}\left/\raisebox{-.2em}{$#2$}\right.}}

\hypersetup{
  pdftitle={},
  pdfauthor={},
  pdfsubject={},
  pdfkeywords={},
  colorlinks=true,
  breaklinks=true,
  bookmarksopen=true,
  bookmarksnumbered=true,
  pdfpagemode=UseOutlines,
  plainpages=false,
  linkcolor=black,
  citecolor=black,
  anchorcolor=black,
  urlcolor  = black
  }
\begin{document}

\title{GLS homogenization tilde map}

\author{Fayadh Kadhem}
\address{College of Arts and Sciences\\
American University of Bahrain\\
Riffa, Bahrain}
\email{fayadh.kadhem@aubh.edu.bh}

\date{\today}
\maketitle

\begin{abstract}
In the construction of a cluster algebra on the homogeneous coordinate ring of a partial flag variety by Gei{\ss}, Leclerc and Schr{\"{o}}er, they defined a special map denoted by ``tilde". This map lifts each element $f$ of the coordinate ring of a Schubert cell uniquely to an element $\widetilde{f}$ of the (multi-homogeneous) coordinate ring of the corresponding partial flag variety. The significance of this map appears from its essential role; it lifts the cluster algebra of the coordinate ring of a cell to a cluster algebra living in the coordinate ring of the corresponding partial flag variety. This paper takes a closer look at this map and gives an explicit algorithm to calculate it for the \textit{generalized minors}.
\end{abstract}

\maketitle
\blfootnote{The research of the author has been supported by NSF grant DMS-2131243.}
\section{Introduction}
Cluster algebras were constructed in 2002 by Fomin and Zelevinsky \cite{FZ} and they quickly formed a very active area of mathematics. This is due to its connections and relations with many other areas in mathematics like algebraic geometry, Poisson geometry, knot theory, mathematical physics and integrable systems. In Fomin and Zelevinsky's work, it has been shown that the notion of cluster algebras is strongly related to the notion of total positivity. Consequently, the first explicit relation between cluster algebras and partial flag varieties appeared in the work of Scott \cite{S} in 2006. Two years after, Gei{\ss}, Leclerc and Schr{\"{o}}er ~\cite{GLS} showed that for type $A$ and $D_4$ the coordinate ring of a partial flag variety contains a cluster algebra induced by a cluster algebra structure of the coordinate ring of the Schubert cell. In fact, they showed the equality for type $A$ and the containment for type $D_4$. However, they showed the equality for type $D_4$ after localizing certain elements. Moreover, they remarked that similar work can be used to generalize their results to the other types of $G$. Indeed, they conjectured that the coordinate ring of the Schubert cell is a cluster algebra and this structure can be lifted using a special map, called the ``tilde" map, to a cluster algebra living in the coordinate ring of the corresponding partial flag variety. In 2011, they proved the conjecture that the coordinate ring of the cell is a cluster algebra for all the simply-laced cases in \cite{GLS3}. Later on, in \cite{GY}, Goodearl and Yakimov constructed a large class of cluster algebras coming from Poisson algebras. As a consequence of this work, the same conjecture was proved for any semisimple complex algebraic group $G$. The same tilde map of \cite{GLS}, was used in \cite{Kadhem}, to show that the cluster algebra of the coordinate ring of the cell can be lifted to a cluster structure living in the coordinate ring of the corresponding partial flag variety for any semisimple complex algebraic group $G$. Moreover, the same paper proved the equality after localizing some certain elements for any type of $G$.\\

In this paper, we take a closer look at the tilde map and give an explicit algorithm to calculate it for the \textit{generalized minors}. This algorithm gives a full description of the cluster algebra of the partial flag variety since the initial seed of the cluster algebra of Goodearl and Yakimov consisted of generalized minors only. The paper is organized as follows: In section 2, we take an overview of cluster algebras by giving the main definitions and notions related to them. After that, we capture the needed notions and results from the Schubert cells and partial flag varieties. In section 4, we describe the tilde map on the generalized minors and then we show how this can be applied to the cluster algebra of \cite{Kadhem} in section 5. Finally, we give explicit type $A$ and type $B$ examples in section 6 to show how the results of section 5 work in both of them.

\section{Cluster algebras}
We start this section by introducing the notion of a cluster algebra. This is given by the following sequence of definitions.
\begin{definition}
A \textit{seed} is a pair $(\textbf{x},B)$ such that $\textbf{x}=(x_1,...,x_n,x_{n+1},...,x_m)$ is a tuple of algebraically independent variables generating a field isomorphic to the field $\mathbb{C}(x_1,...,x_n,x_{n+1},...,x_m)$. Also, $B$ is an $m \times n$ extended skew-symmetrizable matrix, that is, a matrix whose northwestern $n \times n$ submatrix can be transformed to a skew-symmetric matrix by multiplying each row $r_i$ by a nonzero integer $d_i$. The matrix $B$ is called the \textit{exchange matrix} and the tuple $\textbf{x}$ is called the \textit{extended cluster}. The variables $x_1,...,x_n$ are called \textit{mutable}, while the variables $x_{n+1},...,x_m$ are called \textit{frozen}.
\end{definition}

\begin{definition}\label{mutation}
Let $k$ be an index of a mutable variable of a seed $(\textbf{x},B)$. A \textit{mutation} at $k$ is a transformation to a new seed $(\textbf{x}',B')$ in which $B'$ is an $m \times n$ matrix whose entries are
\begin{equation}
    \label{eq1}
b'_{ij}=\begin {cases}
-b_{ij}, & \text{if}\ i=k \text{ or } j=k,\\ 
b_{ij}+\dfrac{|b_{ik}|b_{kj} + b_{ik}|b_{kj}|}{2}, & \text{otherwise};\\
\end{cases}
\end{equation}
and
$\textbf{x}'=(x_1',...,x_n',x_{n+1}',...,x_m')$ is a tuple such that $x_i'=x_i$ for $i \neq k$ and
$$x_k x_k' = \prod_{b_{ik}>0}x_i^{b_{ik}} + \prod_{b_{ik}<0}x_i^{-b_{ik}}.$$
The seed $(\textbf{x}',B')$ obtained by a mutation at $k$ is denoted sometimes by $\mu_k (\textbf{x},B)$. 
\end{definition}

\begin{remark}
It is not hard to see that the mutation of a seed provides a new seed. Moreover, mutating twice at the same index brings the original seed back. In symbols,
$$\mu_k (\mu_k (\textbf{x},B))=(\textbf{x},B).$$
\end{remark}

\begin{definition}
Let $(\textbf{x},B)$ be a seed. A \textit{cluster algebra (of geometric type)} attached to $(\textbf{x},B)$ is the polynomial algebra $\mathcal{A}=\mathbb{C}[x_{n+1},...,x_{m}][\chi]$, where $\chi$ is the set of all possible mutable variables, that is, the mutable variables of the original seed or a seed obtained by a mutation or a sequence of mutations. The seed $(\textbf{x},B)$ is called the \textit{initial seed}.
\end{definition}

\begin{remark}
By the properties of mutation and algebraically independent sets, it is not hard to see that the cluster algebra attached to some seed is the same cluster algebra attached to any mutation of it.
\end{remark}

\begin{definition}
The \textit{rank} of a seed or a cluster algebra attached to it is the number of mutable variables of its initial seed. A cluster algebra is \textit{of finite type} if it has finitely many seeds. Otherwise, it is \textit{of infinite type}.
\end{definition}
\section{Schubert cells and partial flag varieties}
Throughout this paper, let $G$ be a simply-connected semisimple algebraic complex group and let $I$ denote the Dynkin diagram vertex set of $G$. Also, let $K$ denote a subgroup of the set $I$ whose complement is $J=I \setminus K$. For every $G$ denote by $B$ and $B^-$ a pair of Borel and opposite Borel subgroups in which their unipotent radicals are $N$ and $N^-$ respectively. This first section follows \cite{GLS,GLS2}.

\begin{remark}
For every unipotent subgroup there are one-parameter root subgroups indexed by the Dynkin diagram set that form a set of distinguished generators of it. Let $x_i(t)$ $(i \in I,\text{ } t\in \mathbb C)$ denote the ones of $N$ and $y_i(t)$ denote those of $N^-$.
\end{remark}

\begin{example}
If $G=SL_{n+1}$ then $x_i(t)= I + tE_{i,i+1}$, where $I$ is the identity matrix and $E_{i,j}$ is the matrix having 1 in the $i\times j$ entry and 0 elsewhere.
\end{example}

\begin{definition}
A \textit{parabolic subgroup} $P$ of $G$ is a closed subgroup that contains a Borel subgroup.
\end{definition}

\begin{example}
\begin{enumerate}
    \item Any Borel subgroup of $G$ is parabolic.
    \item For the pair $B$ and $B^-$ of opposite Borel subgroups, we define $P_K$ to be the subgroup generated by $B$ and $y_k(t)$ $(k\in K)$. The subgroup $P_K$ is parabolic and called the \textit{standard parabolic subgroup} associated to $B$.
    Similarly, $P_K^-$ is defined to be the subgroup generated by $B^-$ and $x_k(t)$ $(k\in K)$ and is parabolic as well.
\end{enumerate}
\end{example}

\begin{remark}
Every parabolic subgroup is conjugate to a unique standard Borel subgroup. This, in many situations, reduces the study of parabolic subgroups to the study of standard parabolic case.
\end{remark}

\begin{definition}
A \textit{(partial) flag variety} is a quotient of the form $G/P$, where $P$ is a parabolic subgroup of $G$.
\end{definition}

\begin{remark}[Generalized minors]
If $G$ is of type $A$, then a \textit{(flag) minor} is a regular irreducible function of $\mathbb C [G]$ defined as follows: For any $I \subset [1,n]:= \{1,...,n\}$ and any matrix $x \in G$, the minor
$\Delta_I(x)$ is defined to be the determinant of the submatrix of $x$ whose rows are indexed by $I$ and columns are indexed by $1,...,|I|$. This notion was generalized by Fomin and Zelevinsky in \cite{FZ1} to the notion of \textit{(generalized) minor} $\Delta_{u{\varpi_j, w(\varpi_j)}}$, where $u,w$ belong to the Weyl group $W$. The notions of flag minors and generalized minors coincide in type $A$. However, the generalized minor notion makes sense in any type.
\end{remark}
\vspace{10pt}

For each simple reflection $s_i \in W$, let
$$\overline{s_i}:=\exp(f_i) \exp(-e_i) \exp (f_i) \quad \textnormal{ and } \quad \overline{\overline{s_i}}:=\exp(-f_i) \exp(e_i) \exp (-f_i).$$
If $w=s_{i_1}...s_{i_r}$ with $r$ being the length of $w$, then define
$$\overline{w}=\overline{s_{i_1}}...\overline{s_{i_r}} \quad \textnormal{ and } \quad \overline{\overline{w}}=\overline{\overline{s_{i_1}}}...\overline{\overline{s_{i_r}}}.$$
Let $G_0=N^-HN$ be the open set of $G$ consisting of elements having Gaussian decomposition. Indeed, each $x\in G_0$ can be uniquely represented as
$$x=[x]_{-}[x]_0[x]_+,$$
where $[x]_{-}\in N^{-},$ $[x]_0\in H,$ $[x]_+ \in N$.
Fix a choice $\{\varpi_i \mid i \in I \} \subset \mathfrak{h}^*$ of fundamental weights, that is,
$$\varpi_i(h_j)=\delta_{ij}, \quad (i,j \in I).$$
Let $x \mapsto x^{\varpi_i}$ be the corresponding character of $H$. There is a unique regular function $\Delta_{\varpi_i,\varpi_i}$ on $G$ such that
$$\Delta_{\varpi_i,\varpi_i}(x)=[x]_0^{\varpi_i}.$$
Moreover, $G_0=\{x \in G \mid \Delta_{\varpi_i,\varpi_i}(x) \neq 0\ , i \in I\}$. This gives the following definition introduced by Fomin and Zelevinsky in \cite{FZ1}.\\

\begin{definition}
For $u,v \in W$ and $i \in I$, define the \textit{generalized minor} to be the regular function on $G$ given by
$$\Delta_{u\varpi_i, v\varpi_i}(x)=\Delta_{\varpi_i,\varpi_i}(\overline{\overline{u^{-1}}}x \overline{v}).$$
\end{definition}

\begin{remark}
Let $\Pi_J \cong \mathbb N^J$ be the set of elements of the form $\lambda=\sum_{j \in J} a_j \varpi_j$, where $a_j \in \mathbb N$. The homogeneous coordinate ring $\mathbb C [G/P_K^-]$ is a $\Pi_J$ graded module defined as
$$\mathbb C [G/P_K^-]= \bigoplus_{\lambda \in \Pi_J} L(\lambda),$$
where $L(\lambda)$ is the irreducible $G$-module whose highest weight is $\lambda$ (see 2.3 in \cite{GLS}). Moreover, $\mathbb C [G/P_K^-]$ is generated by the subspaces $\bigoplus_{j \in J} L(\varpi_j)$.
\end{remark}

\begin{remark}
  The muti-degree of a generalized minor $\Delta_{\varpi_j, w(\varpi_j)}$ in $\mathbb C [G/P_K^-]$ is $\varpi_j$. It is known that the affine coordinate ring of the unipotent radical cell $N_K$ of $P_K$ is the homogeneous elements of degree 0 of the localization of the homogeneous coordinate ring of the corresponding flag variety by the elements $\Delta_{\varpi_j,\varpi_j}$ where $j \in J$. In symbols,
  \begin{equation}\label{localization}
  \CC[N_K]=\left\{\dfrac{f}{\prod_{j \in J} \Delta_{\varpi_j,\varpi_j}^{a_j}} \mid f\in L \bigg(\sum_{j\in J} a_j \varpi_j \bigg) \right\}.
  \end{equation}
  Equivalently, the generalized minors relate the affine coordinate ring of the cell $N_K$ to the homogeneous coordinate ring of the corresponding flag variety by
  \begin{equation}\label{quotient}
  \mathbb C [N_{K}]=\bigslant{{ \mathbb C [G/P_{K}^{-}]}}{(\Delta_{\varpi_j,\varpi_j}-1)}_{j \in J}.
  \end{equation}
 \end{remark}
 
\begin{remark}
Naturally, there is a canonical projection map
$$\text{proj}_J: \mathbb C [G/P_K^-] \to \mathbb C [N_K]$$
given by the quotient by the ideal generated by $\Delta_{\varpi_j, w(\varpi_j)} -1,$ $(j \in J)$. Remarkably, the restriction of the projection map to each homogeneous $L(\lambda),$ $(\lambda \in \Pi_J)$ gives an injection $L(\lambda) \xhookrightarrow{} \mathbb C[N_K]$.
\end{remark}

\begin{remark}
There is a standard partial ordering $\preceq$ on $\Pi_J$ given by:
$$\lambda \preceq \mu \iff \mu - \lambda \in \mathbb N \{\varpi_j \mid j \in J \},$$
that is, $\mu - \lambda$ is an $\mathbb N$-linear combination of the fundamental weights $\varpi_j$, where $j \in J$.
\end{remark}

\begin{remark}
The Chevalley generators of the Lie algebra $\mathfrak{g}$ of $G$ are denoted, as usual, by $e_i,f_i,h_i$, where $i$ runs over $I$. The $e_i$'s generate $\textnormal{Lie}(N)= \mathfrak{n}.$ Naturally, $N$ acts from the left and right on $\CC[N]$ by these left and right actions:
$$(x \cdot f)(n) = f(nx), \quad  (f \in \CC[N] \textnormal{ and } x,n \in N),$$
$$(f \cdot x)(n) = f(xn), \quad  (f \in \CC[N] \textnormal{ and } x,n \in N).$$
Now, if we differentiate these two actions, we get left and right actions of $\mathfrak{n}$ on the coordinate ring $\CC[N]$. Throughout this paper, the right action of $e_i$ on $f \in \CC[N]$ will be denoted by $e ^ \dagger _i (f) := f \cdot e_i.$
\end{remark}

\begin{lemma}[GLS tilde map] \label{GLS tilde map}
For any element $f$ of the coordinate ring $\mathbb C[N_K]$ there is a unique homogeneous element $\widetilde{f}$ in $\mathbb C [G/P_{K}^{-}]$ whose projection to $\mathbb C[N_K]$ is $f$ and whose multi-degree is minimal with respect to the partial ordering $\preceq$ of weights. 
\end{lemma}

\begin{proof}[Proof (GLS)]
For any $\lambda = \sum_{i \in I} a_i \varpi_i$, it is remarked in 2.5 in \cite{GLS} that the subspace $\text{proj}_I (L(\lambda))$ of $\mathbb C[N]$ can be described as
$$\text{proj}_I (L(\lambda)) = \{ f\in \mathbb C [N] \mid (e_i^{\dagger})^{a_{i}+1}f=0, \text{ } i \in I \}.$$
This implies that $\mathbb C[N_K]$ can be identified with
$$\{ f\in \mathbb C [N] \mid e_k^{\dagger}f=0, \text{ } k \in K \} \subset \mathbb C[N].$$
Thus, for any $\lambda = \sum_{j \in J} a_j \varpi_j \in \Pi_J$, it follows that
$$\text{proj}_J (L(\lambda)) = \{ f\in \mathbb C [N_K] \mid (e_j^{\dagger})^{a_{j}+1}f=0, \text{ } j \in J \}.$$
Now, for $f \in \mathbb C[N_K]$, define
\begin{equation}\label{eq1}
a_j(f):= \max \left\{ s \mid (e_j^{\dagger})^sf \neq 0 \right\},
\end{equation}
and
\begin{equation}\label{eq2}
\lambda(f) := \sum_{j \in J} a_j(f) \varpi_j.
\end{equation}
Obviously, $f$ is an element of $\text{proj}_J \Big(L \big(\lambda(f) \big) \Big)$ where $\lambda(f)$ is minimal with respect to $\preceq$, that is, if $\lambda \in \Pi_J$ such that $f \in \text{proj}_J (L(\lambda))$, then $\lambda(f) \preceq \lambda$. On the other hand, as the restriction of $\text{proj}_J$ to the piece $L \big(\lambda(f) \big)$ is injective, this gives the desired uniqueness. That is, the element $\widetilde{f} \in L \big(\lambda(f) \big)$, whose projection $\proj_J (\widetilde{f})=f$.
\end{proof}

\begin{definition}
For any $f\in \CC[N_K]$, the element $\widetilde{f} \in \CC[G/P_K^-]$ will be called the \textit{lift (or the lifting) of $f$} to $\CC[G/P_K^-].$
\end{definition}

\begin{lemma}\label{multip.}
For any $f,g \in \CC[N_K]$, the lifting commutes with usual multiplication, that is, $\widetilde{f \cdot g}=\widetilde{f} \cdot \widetilde{g}$. In other words, the following diagram
\[ \begin{tikzcd}
\CC[N_K]\times \CC[N_K] \arrow{r}{m_{\CC[N_K]}} \arrow[swap]{d}{\widetilde{\cdot}\times \widetilde{\cdot}} & \CC[N_K] \arrow{d}{\widetilde{\cdot}} \\%
\CC[G/P_K^-]\times \CC[G/P_K^-] \arrow{r}{m_{\CC[G/P_K^-]}}& \CC[G/P_K^-]
\end{tikzcd}
\]
given by
\[
\begin{tikzcd}
f \times g \arrow[mapsto]{r} \arrow[mapsto]{d}
& f \cdot g  \arrow[mapsto]{d}\\
\widetilde{f} \times \widetilde{g} \arrow[mapsto]{r}
& \widetilde{f \cdot g}=\widetilde{f} \cdot \widetilde{g}
\end{tikzcd}
\]
commutes. Moreover, if $a_j(f+g)= \max \{a_j(f), a_j(g) \}$ for all $j \in J$, then
$$\widetilde{f+g}=\mu \widetilde{f} + \nu \widetilde{g},$$
where $\mu$ and $\nu$ are relatively prime monomials in the variables $\Delta_{\varpi_j, \varpi_j}$, $(j \in J)$.
\end{lemma}
\begin{proof}[Proof (GLS)]
The first property follows easily by Leibniz formula and the fact that the endomorphism $e_j^\dagger$ is a derivation of $\CC[N_K]$, for all $j\in J.$ The additional assumption in the second statement implies the existence of relatively prime monomials $\mu$ and $\nu$ in $\Delta_{\varpi_j, \varpi_j}$'s such that the multi-degree of each of $\mu \widetilde{f}$ and $\nu \widetilde{g}$ is the same as the one of $\widetilde{f+g}$. This completes the proof.
\end{proof}

\begin{definition}
For a generalized minor $\Delta_{\varpi_j, w(\varpi_j)},$ define the \textit{restricted minor} $D_{\varpi_j, w(\varpi_j)}$ to be the restriction of it to $N$.
\end{definition}
\begin{remark}
The restricted minors are the main objects of the cluster structure defined by Goodearl and Yakimov on the coordinate ring of a Schubert cell. Their lift plays a significant role in constructing the cluster structure of a partial flag variety (see \cite{GLS, GLS2, Kadhem}). Naively, one might expect that the lift of a restricted minor is always a generalized minor, but this need not be true in general.
\end{remark}

Gei{\ss}, Leclerc and Schr{\"{o}}er addressed this example of type $A$ in \cite{GLS}:

\begin{example}
Consider $G=SL_6$, a group of type $A_5$. Take $J=\{1,3 \}$, so $K=\{2,4,5 \}$. The restricted minor $D_{13,56}$ cannot be written as $D_{12...m,i_1i_2...i_m}$. Therefore, $\widetilde{D}_{13,56}$ cannot be a flag minor. However,
\begin{align*}
D_{13,56} &=D_{1,2}D_{23,56}-D_{123,256}\\
&=D_{1,2}D_{123,156}-D_{123,256}.
\end{align*}
Hence,
\begin{align*}
\widetilde{D}_{13,56} &=\Delta_{1,2} \Delta_{123,156}- \Delta_{1,1} \Delta_{123,256}\\
&=\Delta_{2} \Delta_{156}- \Delta_{1} \Delta_{256}.
\end{align*}
\end{example}
\section{Lift degree and explicit algorithm}
The main goal of this section is to solve the problem of lifting the minors that occurred at the end of the previous section. We start this section by the following proposition:

\begin{proposition}
Let $f \in \CC[N_K]$. Assume that $f=D_{\varpi_j, w(\varpi_j)}$ and $j \in J$. Then $\widetilde{f}=\Delta_{\varpi_j, w(\varpi_j)}$.
\end{proposition}

\begin{proof}
It is obvious that $\Delta_{\varpi_j, w(\varpi_j)}$ is an element that projects to $D_{\varpi_j, w(\varpi_j)}$ such that its degree is $\varpi_j$. Now, we may assume that there is an element $g \in \CC[G/P_K^-]$ whose projection $\proj_J(g)=D_{\varpi_j, w(\varpi_j)}$ and whose degree $\lambda_g \preceq \varpi_j$. By equation (\ref{eq2}), the only possibility is $\lambda_g=a_j \varpi_j$ and $a_j \leq 1$. This forces $\lambda_g$ to be $\varpi_j$, as it is 0 otherwise. By uniqueness, $g=\Delta_{\varpi_j, w(\varpi_j)}$.
\end{proof}

\begin{lemma}
Assume that $f, g \in \CC[G/P_K^-]$ such that $\textnormal{proj}_J (f)= \textnormal{proj}_J (g)$ and the multi-degree of both is minimal with respect to the partial ordering $\preceq$ of weights. Then $f=g$.
\end{lemma}

\begin{proof}
Using equation (\ref{localization}), there are natural numbers $a_j$ and $b_j$ such that
$$\proj_J (f) = \dfrac{f}{\prod_{j \in J} \Delta_{\varpi_j,\varpi_j}^{a_j}} \quad \quad \textnormal{and} \quad \quad \proj_J (g)=\dfrac{g}{\prod_{j \in J} \Delta_{\varpi_j,\varpi_j}^{b_j}}.$$
The minimality gives exactly two possibilities, either $f=g$ or $\lambda(f)$ and $\lambda(g)$ are incomparable. The uniqueness of the tilde map makes the latter impossible. Hence, $f=g$ and consequently $a_j=b_j$.
\end{proof}
This allows to rewrite equation (\ref{localization}) as:
\begin{equation}\label{localization 2}
  \CC[N_K]=\left\{ \proj_J(f)=\dfrac{f}{\prod_{j \in J} \Delta_{\varpi_j,\varpi_j}^{a_j}} \mid f\in L \bigg(\sum_{j\in J} a_j \varpi_j \bigg) \textnormal{ and $a_j$ is minimal} \right\}.
  \end{equation}
  
\begin{corollary} \label{minor 1}
If $j\in J$, then the restricted minor $D_{\varpi_j,w\varpi_j}$ is given by
$$D_{\varpi_j,w\varpi_j}=\dfrac{\Delta_{\varpi_j,w\varpi_j}}{\Delta_{\varpi_j,\varpi_j}}.$$
\end{corollary}
  
\begin{proposition} \label{GLS tilde map 2}
Let $\dfrac{f}{\prod_{j \in J} \Delta_{\varpi_j,\varpi_j}^{a_j}}$ be an element of $\CC[N_K]$ in which $a_j$ is minimal for each $j$. Then
$$\text{\test{{\Big(\dfrac{f}{\prod_{j \in J} \Delta_{\varpi_j,\varpi_j}^{a_j}}\Big)}}}=f.$$
\end{proposition}
\begin{proof}
By (\ref{localization 2}) it is clear that $f\in L \bigg(\sum_{j\in J} a_j \varpi_j \bigg)$. Obviously, $$\proj_J(f)=\dfrac{f}{\prod_{j \in J} \Delta_{\varpi_j,\varpi_j}^{a_j}}.$$
The result now follows by the minimality of the $a_j$'s and the uniqueness of the GLS tilde map.
\end{proof}


For $f \in \CC[N_K]$, recall the definition of $a_j(f)$ given in \ref{eq1}.

\begin{definition}\label{lift deg}
For a restricted minor $D_{\varpi_{i_n}, w(\varpi_{i_n})}$, define the \textit{lift degree} to be the integer $d_n$ in the equation:
\begin{equation}\label{lift deg eq}
    s_{i_1}(s_{i_2}...s_{i_n})(\varpi_{i_n})=s_{i_2}...s_{i_n}(\varpi_{i_n})-d_{n} \alpha_{i_1},
\end{equation}
where $w=s_{i_1}s_{i_2}...s_{i_n}$ and $\alpha_{i_1}$ is the vertex of the Dynkin diagram indexed by $i_1$.
\end{definition}

\begin{proposition} Assume the setting and notation of Definition (\ref{lift deg}). If $J=\{i_1\}$, then
the lift $\widetilde{D}_{\varpi_{i_n}, w(\varpi_{i_n})}$ of the minor $D_{\varpi_{i_n}, w(\varpi_{i_n})}$ is of degree $d_n \varpi_{i_1}$.
\end{proposition}
\begin{proof}
This follows from the fact that
\begin{align*}
a_l(D_{\varpi_{i_n}, w(\varpi_{i_n})}) &= \max \left\{ s \mid (e_l^{\dagger})^s D_{\varpi_{i_n}, w(\varpi_{i_n})} \neq 0 \right\}\\
&=
\begin{cases}
d_n, \quad &\textnormal{if }{l=i_1}, \\
0, &\textnormal{otherwise;}
\end{cases}
\end{align*}
where $a_l$ is as introduced in Equation (\ref{eq1}). In fact, the first statement is clear and the second follows from the fact that $\Bar{u}x_j(t)\Bar{u}^{-1} \in N$, when $\ell(us_i) > \ell (u)$ and $i \neq j$.
\end{proof}
\begin{corollary} \label{minor 2}
Assume the setting and notation of Definition (\ref{lift deg}). If $J=\{i_1\}$, then there is a unique $f \in \CC[G/P_K^-]$ of homogeneous degree $d_n\varpi_{i_1}$ such that
$$D_{\varpi_{i_n}, w(\varpi_{i_n})}=\dfrac{f}{\Delta_{\varpi_{i_1}, \varpi_{i_1}}^{d_n}}.$$
\end{corollary}

\section{An application to cluster algebras}
We start this section with some setup and then apply the results of the previous section to the cluster algebra of the homogeneous coordinate ring a partial flag variety defined in \cite{Kadhem}.\\

\begin{definition}\label{Cell cluster}
Let $w=s_{i_1}...s_{i_n} \in W$. Define the functions
$$p(k):=\begin{cases}
\textnormal{max}\{j < k \textnormal{ }| \textnormal{ } i_j=i_k \},& \text{if such } j \text{ exists;}\\
- \infty, & \text{otherwise.}
\end{cases}$$
$$s(k):=\begin{cases}
\textnormal{min}\{j > k \textnormal{ }| \textnormal{ } i_j=i_k \},& \text{if such } j \text{ exists;}\\
\infty, & \text{otherwise.}
\end{cases}$$
Also, set $$S(w):=\{i \in I \mid s_i \leq w \}=\{i \in I \mid i=i_k \textnormal{ for some } k \in [1,m] \}.$$
In \cite{Kadhem}, it is shown that the work of Goodearl and Yakimov gives a canonical cluster structure in $\CC[N_K]$
in which its initial exchange matrix $\widetilde{B}^w$ is of size $m \times (m-|S(w)|)$ and its $j \times k$ entry is given by
$$(\widetilde{B}^w)_{jk} = \begin{cases}
1, & \text{if } j=p(k), \\
-1, & \text{if } j=s(k),\\
a_{i_j i_k}, & \text{if } j<k<s(j)<s(k),\\
-a_{i_j i_k}, & \text{if } k<j<s(k)<s(j),\\
0, & \text{otherwise;}
\end{cases}$$
where the entry $a_{i_j i_k}$ is the same $i_j \times i_k$ entry of the Cartan matrix of the same type. The initial cluster variables are $D_{\varpi_{i_k},w_{\leq k} \varpi_{i_k}}$. The frozen ones are exactly those indexed by the $k$'s that satisfy $s(k)=\infty$.
\end{definition}

\begin{remark} \label{cluster mutation}
Let $(\widetilde{{{\textnormal{\textbf{x}}}}},\widetilde{B})$ be a seed of the cluster algebra $\mathcal{A}= \CC[N_K]$. Mutate at $k$ to get the exchange relation
$$x_k x_k' = M_k + L_k,$$
where $M_k,L_k$ are monomials in the variables $x_1,...,x_{k-1},x_{k+1},...,x_n.$ As shown in ~\cite{GLS}, this lifts to the equation
$$\widetilde{x_k x'_k}=\mu_k\widetilde{M_k}+\nu_k \widetilde{L_k},$$
where $\mu_k$ and $\nu_k$ are relatively prime monomials in the vatiables $\Delta_{{\varpi_{j}},{\varpi_{j}}}$, where $j$ runs in $J$. This means that $\mu_k$ and $\nu_k$ can be written as
$$\mu_k = \prod_{j \in J} \Delta_{{\varpi_{j}},{\varpi_{j}}}^{\alpha_j} \quad \textnormal{ and } \quad \nu_k = \prod_{j \in J} \Delta_{ {\varpi_{j}}, {\varpi_{j}}}^{\beta_j},$$
where $\min \{\alpha_j, \beta_j \}=0$ for all $j$.\\

Note here that the monomial $M_k$ is the one corresponding to the positive $b_{ik}$'s in Definition \ref{mutation}, while the monomial $L_k$ is the one corresponding to the negative $b_{ik}$'s.
\end{remark}

\begin{definition} \label{main definition}
Let $(\textnormal{\textbf{x}},{B})$ be a seed of the cluster algebra $\mathcal{A}_J=\CC [N_K]$. Define the pair $(\widehat{\textnormal{\textbf{x}}},\widehat{B})$ as follows:
\begin{itemize}
    \item The tuple $\widehat{\textnormal{\textbf{x}}}$ consists of the variables $\widetilde{x}$ induced by lifting each variable $x$ of $\textnormal{\textbf{x}}$ and consists also of the generalized minors $\Delta _{{\varpi_j}, {\varpi_j}}$ modded out in $\CC[N_K]$. The variables $\widetilde{x}$ preserve the same type of the variables $x$ (mutable or frozen), while the minors $\Delta _{{\varpi_j}, {\varpi_j}}$ are frozen.
    \item The matrix $\widehat{B}$ is given as follows:
Extend the matrix $B$ of the initial seed of $\mathcal{A}_J$ by $|J|$ rows labeled by the elements of $J$ such that the entries are
\[
  \widehat{b}_{jk} =
  \begin{cases}
\beta_j , & \text{if $\beta_j \neq 0$;} \\
- \alpha_j,& \text{else,}
  \end{cases}
  \]
where $\alpha_j$ and $\beta_j$ are as in Remark \ref{cluster mutation}.
\end{itemize}
\end{definition}

\noindent The combination of the results of the previous section together with the results of \cite{Kadhem} gives us the following theorem:

\begin{theorem}\label{main}
Let $(\textnormal{\textbf{x}},{B})$ be the initial seed of the cluster algebra $\mathcal{A}=\CC[N_K]$ described in Definition \ref{Cell cluster}. Then $\mathcal{A}$ lifts to a cluster algebra $\widehat{\mathcal{A}} \subset \CC[G/P_K^-]$ whose initial exchange cluster is $(\widehat{\textnormal{\textbf{x}}},\widehat{B})$. In particular, the extended cluster variables are $\big\{ \widetilde{ D}_{\varpi_{i_k},w_{\leq k} \varpi_{i_k}}  \big\} \sqcup \{ \Delta_{\varpi_j, \varpi_j} \mid j\in J \}$.
If $J=\{i_1\}$, then for each $D_{\varpi_{i_k},w_{\leq k} \varpi_{i_k}}$, there exists a unique $f \in \CC[G/P_K^-]$ of degree $d_k \varpi_{i_1}$ such that
$$\widetilde{ D}_{\varpi_{i_k},w_{\leq k} \varpi_{i_k}}=f.$$
\end{theorem}

\begin{remark}
For a full generality of type $A$, that is, if $J$ is any nonempty subset of $I$, one can use the following fact: If $Q$ is a parabolic subgroup between $P$ and $G$, then
    $$w_0 w_{P,0}^{-1} = (w_0 w_{Q,0}^{-1})   (w_{Q,0} w_{P,0}^{-1})$$
where $w_{P,0}$ and $w_{Q,0}$ are the longest element of the arbitrary parabolic subgroups $P$ and $Q$ respectively. This will be length-additive, meaning that
$$\textnormal{dim} (G/P) = \textnormal{dim} (G/Q) + \textnormal{dim} (Q/P).$$
So, everything can be reduced to the known cases, for example when $P$ is maximal. Another useful idea, for type $A$, is to use the Plucker relation together with the fact that the generalized minor is nothing but a flag minor. For instance, see Example 10.3 in \cite{GLS}.
\end{remark}

\section{Examples}
The results of the previous sections can be used to get these explicit examples.

\begin{example}[Type $A$, c.f. \cite{GLS}]\label{explicit type A}
Let $G$ be a semisimple algebraic group of type $A_4$. That is, $G=SL_6$. Take $J=\{2\}$ and $K=\{1,3,4\}$. We can express the longest word $w_0$ as
$$w_0=s_1s_3s_4s_3\boldsymbol{s_2s_3s_4s_1s_2s_3}.$$
The subword $w_K=s_2s_3s_4s_1s_2s_3$ generates $N_K$. By definition, one can easily see that
$$s(1)=5, \quad s(2)=6, \quad s(3)=s(4)=s(5)=s(6)=\infty.$$
Therefore, the list of extended cluster variables is
\begin{align*}
j=1 &\implies D_{\varpi_{2},s_2 \varpi_{2}}; & \textnormal{(mutable)}\\
j=2 &\implies D_{\varpi_{3},s_2s_3 \varpi_{3}}; & \textnormal{(mutable)}\\
j=3 &\implies D_{\varpi_{4},s_2s_3s_4 \varpi_{4}}; & \textnormal{(frozen)}\\
j=4 &\implies D_{\varpi_{1},s_2s_3s_4s_1\varpi_{1}}; & \textnormal{(frozen)}\\
j=5 &\implies D_{\varpi_{2},s_2s_3s_4s_1s_2\varpi_{2}}; & \textnormal{(frozen)}\\
j=6 &\implies D_{\varpi_{3},s_2s_3s_4s_1s_2s_3 \varpi_{3}}. & \textnormal{(frozen)}\\
\end{align*}
The exchange matrix $B$ is
 \[
{B}=\begin{blockarray}{ccc}
 1 & 2  \\
\begin{block}{(cc)c}
  0 & -1 & 1 \\
  1 & 0 & 2 \\
  \cmidrule(lr){1-2}
  0 & 1 & 3 \\
  1 & 0 & 4 \\
  -1 & 1 & 5 \\
  0 & -1 & 6 \\
\end{block}
\end{blockarray}
\quad .\]
Now, we use the tilde map to lift each extended cluster variable to get the cluster algebra $\widehat{\mathcal{A}} \subset \CC[G/P_K^-]$ that has the following extended cluster variables
\begin{align*}
& \widetilde{D}_{\varpi_{2},s_2 \varpi_{2}}; & \textnormal{(mutable)}\\
& \widetilde{D}_{\varpi_{3},s_2s_3 \varpi_{3}}; & \textnormal{(mutable)}\\
& \widetilde{D}_{\varpi_{4},s_2s_3s_4 \varpi_{4}}; & \textnormal{(frozen)}\\
& \widetilde{D}_{\varpi_{1},s_2s_3s_4s_1\varpi_{1}}; & \textnormal{(frozen)}\\
& \widetilde{D}_{\varpi_{2},s_2s_3s_4s_1s_2\varpi_{2}}; & \textnormal{(frozen)}\\
& \widetilde{D}_{\varpi_{3},s_2s_3s_4s_1s_2s_3 \varpi_{3}}. & \textnormal{(frozen)}\\
& \Delta_{\varpi_{2},\varpi_{2}}. & \textnormal{(frozen)}\\
\end{align*}

Also, from Corollary \ref{minor 1} and Corollary \ref{minor 2} we get that

\begin{align*}
& \widetilde{D}_{\varpi_{2},s_2 \varpi_{2}} & =& \Delta_{\varpi_{2},s_2 \varpi_{2}}\\
& \widetilde{D}_{\varpi_{3},s_2s_3 \varpi_{3}} & =& f_1\\
& \widetilde{D}_{\varpi_{4},s_2s_3s_4 \varpi_{4}}& =& f_2\\
& \widetilde{D}_{\varpi_{1},s_2s_3s_4s_1\varpi_{1}}& =& f_3\\
& \widetilde{D}_{\varpi_{2},s_2s_3s_4s_1s_2\varpi_{2}}& =& \Delta_{\varpi_{2},s_2s_3s_4s_1s_2\varpi_{2}}\\
& \widetilde{D}_{\varpi_{3},s_2s_3s_4s_1s_2s_3 \varpi_{3}}& =& f_4\\
\end{align*}

\noindent such that $f_1, f_2, f_3, f_4$ are the appropriate unique functions of Corollary \ref{minor 2}.

Let us now calculate the degree of each of them. For $f_1$, consider the equation
\begin{align*}
s_2(s_3\varpi_3)&=s_3\varpi_3-(\alpha_2^\vee,s_3\varpi_3)\alpha_2\\
&=s_3\varpi_3-(\alpha_2^\vee,\varpi_3-\alpha_3)\alpha_2\\
&=s_2\varpi_2-\alpha_2.
\end{align*}
Hence, the left degree is 1 and by Corollary \ref{minor 2} the homogeneous degree of $f_1$ is $\varpi_2$.\\

\noindent Consequently,
$$f_1=\dfrac{\Delta_{\varpi_{3},s_2s_3 \varpi_{3}}\Delta_{\varpi_{2}, \varpi_{2}}}{\Delta_{\varpi_{3}, \varpi_{3}}}$$
Analogously, one can get that the homogeneous degree of $f_2,f_3,f_4$ is $\varpi_2$ as well. Therefore,
\begin{align*}
    f_2 &= \dfrac{\Delta_{\varpi_{4},s_2s_3s_4 \varpi_{4}}\Delta_{\varpi_{2}, \varpi_{2}}}{\Delta_{\varpi_{4}, \varpi_{4}}} ;\\
    f_3 &= \dfrac{\Delta_{\varpi_{4},s_2s_3s_4s_1 \varpi_{1}}\Delta_{\varpi_{2}, \varpi_{2}}}{\Delta_{\varpi_{1}, \varpi_{1}}} ;\\
    f_4 &= \dfrac{\Delta_{\varpi_{3},s_2s_3s_4s_1s_2s_3 \varpi_{3}}\Delta_{\varpi_{2}, \varpi_{2}}}{{\Delta_{\varpi_{3}, \varpi_{3}}}} .\\
\end{align*}

\noindent Now, the lift of the exchange relations can be calculated easily. For instance, mutating at $k=1$, we get
\begin{align*}
    D_{\varpi_{2},s_2 \varpi_{2}}D_1'
    &=D_{\varpi_{3},s_2s_3 \varpi_{3}}D_{\varpi_{1},s_2s_3s_4s_1\varpi_{1}}+D_{\varpi_{2},s_2s_3s_4s_1s_2\varpi_{2}}\\
    &=\frac{f_1}{\Delta_{\varpi_2,\varpi_2}}\frac{f_3}{\Delta_{\varpi_2,\varpi_2}}+\frac{\Delta_{\varpi_{2},s_2s_3s_4s_1s_2\varpi_{2}}}{\Delta_{\varpi_2,\varpi_2}}\\
    &=\frac{f_1f_3+\Delta_{\varpi_2,\varpi_2}\Delta_{\varpi_{2},s_2s_3s_4s_1s_2\varpi_{2}}}{\Delta_{\varpi_2,\varpi_2}^2}.
\end{align*}
Thus, by Lemma \ref{multip.} and Proposition \ref{GLS tilde map 2}
$$\text{ \test{D_{\varpi_{2},s_2 \varpi_{2}}D_1'}}=f_1f_3+\Delta_{\varpi_2,\varpi_2}\Delta_{\varpi_{2},s_2s_3s_4s_1s_2\varpi_{2}}.$$
The other mutation exchange relations can be lifted similarly. Finally, we get that the exchange matrix attached to the cluster algebra $\widehat{\mathcal{A}} \subset \CC[G/P_K^-]$ is

\[
\widehat{B}=\begin{blockarray}{ccc}
 1 & 2  \\
\begin{block}{(cc)c}
  0 & -1 & 1 \\
  1 & 0 & 2 \\
  \cmidrule(lr){1-2}
  0 & 1 & 3 \\
  1 & 0 & 4 \\
  -1 & 1 & 5 \\
  0 & -1 & 6 \\
  \cmidrule(lr){1-2}
  -1 & 0 & 1 \in J \\
\end{block}
\end{blockarray}
\quad .\]

\end{example}








\begin{example}[Type $B$, c.f. \cite{Kadhem}]\label{explicit type B}
Let $G$ be a semisimple algebraic group of type $B_3$, say $G={SO}_{2(3)+1}={SO}_7$, $J=\{3\}$ and $K=I \setminus J=\{1,2\}.$ Consider the longest word
\[w_0={s_1s_2s_1} \boldsymbol{s_3s_2s_1s_3s_2s_3}.\]
In \cite{Kadhem}, it was shown that $\CC[G/P_K^-]$ has a cluster structure whose initial extended cluster is given by the variables
\begin{align*}
&\widetilde{D}_{\varpi_{3},s_3 \varpi_{3}}; &\textnormal{(mutable)}\\
&\widetilde{D}_{\varpi_{2},s_3s_2 \varpi_{2}}; &\textnormal{(mutable)}\\
&\widetilde{D}_{\varpi_{3},s_3s_2s_1s_3 \varpi_{3}}; &\textnormal{(mutable)}\\
&\widetilde{D}_{\varpi_{1},s_3s_2s_1 \varpi_{1}}; &\textnormal{(frozen)}\\ 
&\widetilde{D}_{\varpi_{2},s_3s_2s_1s_3s_2 \varpi_{2}}; &\textnormal{(frozen)}\\
&\widetilde{D}_{\varpi_{3},s_3s_2s_1s_3s_2s_3 \varpi_{3}}; &\textnormal{(frozen)}\\
&\Delta_{\varpi_3,\varpi_3}. &\textnormal{(frozen)}
\end{align*}
Now, using Corollary \ref{minor 1} and Corollary \ref{minor 2} it is easily seen that
\begin{align*}
&\widetilde{D}_{\varpi_{3},s_3 \varpi_{3}} &=&\Delta_{\varpi_{3},s_3 \varpi_{3}}\\
&\widetilde{D}_{\varpi_{2},s_3s_2 \varpi_{2}}
&=&f\\
&\widetilde{D}_{\varpi_{3},s_3s_2s_1s_3 \varpi_{3}} &=&\Delta_{\varpi_{3},s_3s_2s_1s_3 \varpi_{3}}\\
&\widetilde{D}_{\varpi_{1},s_3s_2s_1 \varpi_{1}}
&=&g\\ 
&\widetilde{D}_{\varpi_{2},s_3s_2s_1s_3s_2 \varpi_{2}}  &=&h\\
&\widetilde{D}_{\varpi_{3},s_3s_2s_1s_3s_2s_3 \varpi_{3}}
&=&\Delta_{\varpi_{3},s_3s_2s_1s_3s_2s_3 \varpi_{3}},
\end{align*}
where $f,g,h$ are the unique suitable functions gotten by Corollary \ref{minor 2}.
Moreover, the lifting of the mutation exchange relations of the cluster algebra $\mathcal{A}$ becomes obvious now. In fact,
exchange relation induced by mutating at $1$ is
\begin{align*}
{D}_{\varpi_{3},s_3 \varpi_{3}} D_1' &= {D}_{\varpi_{2},s_3s_2 \varpi_{2}} + {D}_{\varpi_{3},s_3s_2s_1s_3 \varpi_{3}}.\\
\end{align*}
Now, to find the lift degree of ${D}_{\varpi_{2},s_3s_2 \varpi_{2}}$ we consider the equation
$$s_3(s_2(\varpi_2)) = s_2(\varpi_2) - 2 \alpha_3.$$
This shows that the lift degree of ${D}_{\varpi_{2},s_3s_2 \varpi_{2}}$ is 2. Thus, it follows that the lift degree of $f \in \CC[G/P_K^-]$ is $2\varpi_3$. Moreover, by Corollary \ref{minor 2}
$$f=\dfrac{\Delta_{\varpi_{2},s_3s_2 \varpi_{2}}\Delta_{\varpi_3,\varpi_3}^2}{\Delta_{\varpi_2,\varpi_2}} $$
and
\begin{align*}
{D}_{\varpi_{3},s_3 \varpi_{3}} D_1' &= {D}_{\varpi_{2},s_3s_2 \varpi_{2}} + {D}_{\varpi_{3},s_3s_2s_1s_3 \varpi_{3}}\\
&= \dfrac{f}{{\Delta}_{\varpi_{3}, \varpi_{3}}^2} + \dfrac{{\Delta}_{\varpi_{3},s_3s_2s_1s_3 \varpi_{2}}}{{\Delta}_{\varpi_{3}, \varpi_{3}}}\\
&=\dfrac{ f+ \Delta_{\varpi_{3}, \varpi_{3}} \cdot {\Delta}_{\varpi_{3},s_3s_2s_1s_3 \varpi_{2}}}{\Delta_{\varpi_{3}, \varpi_{3}}^2}.
\end{align*}
Hence, by Lemma \ref{multip.} and Proposition \ref{GLS tilde map 2}
$$ 
\text{  \test{{D}_{\varpi_{3},s_3 \varpi_{3}} D_1'}}=\widetilde{D}_{\varpi_{3},s_3 \varpi_{3}} \widetilde{D_1'}=f+ \Delta_{\varpi_{3}, \varpi_{3}} \cdot {\Delta}_{\varpi_{3},s_3s_2s_1s_3 \varpi_{2}}.
$$
On the other hand, the exchange relation induced by mutating at 2 is
\begin{align*}
{D}_{\varpi_{2},s_3s_2 \varpi_{2}} D_2' &= {D}_{\varpi_{3},s_3s_2s_1s_3 \varpi_{3}}^2 {D}_{\varpi_{1},s_3s_2s_1 \varpi_{1}} + {D}_{\varpi_{3},s_3\varpi_{3}}^2{D}_{\varpi_{2},s_3s_2s_1s_2\varpi_{2}}.\\
\end{align*}
Now, it is not hard to see that
$$s_3(s_2s_1(\varpi_1)) = s_2s_1(\varpi_1) - 2 \alpha_1 \quad \text{and} \quad
    s_3(s_2s_1s_3s_2(\varpi_2)) = s_2s_1s_3s_2(\varpi_2) -2 \alpha_3$$
Therefore, the lift degree of both $g$ and $h$ is $2\varpi_3$ again. Hence, by Corollary \ref{minor 2}, again
\begin{align*}
    g &=\dfrac{\Delta_{\varpi_{1},s_3s_2s_1 \varpi_{1}}\Delta_{\varpi_3,\varpi_3}^2}{\Delta_{\varpi_1,\varpi_1}}\\
    h &=\dfrac{\Delta_{\varpi_{2},s_3s_2s_1s_3s_2 \varpi_{2}}\Delta_{\varpi_3,\varpi_3}^2}{\Delta_{\varpi_2,\varpi_2}}
    \end{align*}
and
\begin{align*}
{D}_{\varpi_{2},s_3s_2 \varpi_{2}} D_2' &= {D}_{\varpi_{3},s_3s_2s_1s_3 \varpi_{3}}^2 {D}_{\varpi_{1},s_3s_2s_1 \varpi_{1}} + {D}_{\varpi_{3},s_3\varpi_{3}}^2{D}_{\varpi_{2},s_3s_2s_1s_2\varpi_{2}}\\
&=
 \dfrac{{\Delta}_{\varpi_{3},s_3s_2s_1s_3 \varpi_{3}}^2}{\Delta_{\varpi_{3}, \varpi_{3}}^2} \cdot \dfrac{g}{\Delta_{\varpi_{3}, \varpi_{3}}^2} + \dfrac{{\Delta}_{\varpi_{3},s_3 \varpi_{3}}^2}{\Delta_{\varpi_{3}, \varpi_{3}}^2}\cdot \dfrac{h}{\Delta_{\varpi_{3}, \varpi_{3}}^2}\\
 &=\dfrac{{\Delta}_{\varpi_{3},s_3s_2s_1s_3 \varpi_{3}}^2 \cdot {g} + {{\Delta}_{\varpi_{3},s_3 \varpi_{3}}^2} \cdot {h}}{\Delta_{\varpi_{3}, \varpi_{3}}^4}
\end{align*}
Hence, by Lemma \ref{multip.} and Proposition \ref{GLS tilde map 2} again
$$ 
\text{\test{{D}_{\varpi_{2},s_3s_2 \varpi_{2}} D_2'}}=\widetilde{D}_{\varpi_{2},s_3s_2 \varpi_{2}} \widetilde{D_2'}={\Delta}_{\varpi_{3},s_3s_2s_1s_3 \varpi_{3}}^2 \cdot {g} + {{\Delta}_{\varpi_{3},s_3 \varpi_{3}}^2} \cdot {h}.$$
The mutation relation induced by mutating at 3 lifts similarly.\\
The initial exchange matrix of the cluster algebra $\CC[N_K]$ given in Definition \ref{Cell cluster} is
\[B=
\begin{blockarray}{cccc}
1 & 2 & 4 \\
\begin{block}{(ccc)c}
  0 & -2 & 1 & 1 \\
  1 & 0 & -1 & 2 \\
  -1 & 2 & 0 & 4 \\
  0 & 1 & 0 & 3 \\
  0 & -1 & 1 & 5 \\
  0 & 0 & -1 & 6\\
\end{block}
\end{blockarray}
 .\]
\end{example}
Using the calculations above and Theorem \ref{main}, we get that $\CC[G/P_K^-]$ contains the cluster algebra $\widehat{\mathcal{A}}$ whose initial extended cluster variables are the ones in the list above and whose initial exchange matrix is
\[
\widehat{B}=\begin{blockarray}{cccc}
 1 & 2 & 4 \\
\begin{block}{(ccc)c}
  0 & -2 & 1 & 1 \\
  1 & 0 & -1 & 2 \\
  -1 & 2 & 0 & 4 \\
  0 & 1 & 0 & 3 \\
  0 & -1 & 1 & 5 \\
  0 & 0 & -1 & 6\\
  \cmidrule(lr){1-3}
  -1 & 0 & 0 & j \in J\\
\end{block}
\end{blockarray}
.\]

\begin{remark}
Note that our goal once we find the additional rows of $\widehat{B}$ is to make the relations homogeneous. Observing the previous examples, we can easily see that if $J=\{i_1\}$, the last row of the matrix $\widehat{B}$ is obtained by making the summation of the degrees of the positive monomial and the negative monomial equal in each column. Hence, the entries of the last row are obtained according to that.
\end{remark}

\section{General theorems}
In the previous section, we saw how one uses the results of this paper to obtain explicit examples of type $A$ and type $B$. In this section, we generalize these examples to any 2-step parabolic subgroup of type $A$ and to any maximal parabolic subgroup of type $B$. Other general examples can be obtained similarly to the combination of the results of this section and the previous one.

\begin{example}[2-step parabolic subgroups of type $A$]
Let $G$ be of type $A_n$, that is, $G=SL_{n+1}$. Take $J=\{j_1,j_2\}$ with $j_1 < j_2$ and $K=\{1,...,n\} \setminus \{j_1,j_2\}$. The longest word can be determined once we know the positions of $j_1$ and $j_2$. Indeed, the desired expression of the longest word starts with the longest word of $A_{j_1-1}$ indexed by $1,...,j_1-1$, followed by the longest word of $A_{j_2-j_1-1}$ indexed by $j_1+1,...,j_2-1$ and then followed by the longest word of $A_{n-j_2}$ indexed by $j_2+1,...,n$, followed by the completion of that longest word by a subword generating $N_K$. More concretely, let
\begin{align*}
u_1&=s_1s_2...s_{j_1-1}s_1s_2...s_{j_1-2}...s_1s_2s_1 \\
u_2&=s_{j_1+1}s_{j_1+2}...s_{j_2-1}s_{j_1+1}s_{j_1+2}...s_{j_2-2}...s_{j_1+1}s_{j_1+2}s_{j_1+1}\\
u_3&=s_{j_2+1}s_{j_2+2}...s_{n-1}s_{j_2+1}s_{j_2+2}...s_{n-2}...s_{j_2+1}s_{j_2+2}s_{j_2+1} 
\end{align*}
Then, there exists a subword $u_4$ in which the longest word $w_0$ can be expressed as $w_0=u_1u_2u_3u_4$. One way to construct $u_4$ is to take $s_1...s_n$, together with a reduced word consisting of $j_2-1$ reduced subwords each of length $n-j_2$, followed by a reduced word consisting $j_2-j_1$ reduced subwords each of length $j_1$. In particular, we may choose $u_4$ to be
\begin{align*}
u_4 &=s_1...s_n u_5 u_6,\\
u_5 &=s_1...s_{n-j_2}s_2...s_{n+1-j_2}s_3...s_{n+2-j_2}...s_{j_2-1}...s_{n-2},\\
u_6 &=s_1...s_{j_1}s_2...s_{j_1+1}s_3...s_{j_1+2}...s_{j_2-j_1}...s_{j_2-1}.
\end{align*}
In fact, this $u_4$ generates $N_K$ and it is of length $n+(n-j_2)(j_2-1)+j_1(j_2-j_1)$.\\

Obviously, one now can use the fact that the generalized minor is a flag minor in type $A$ and use the Plucker relation. Then, follow the procedure of Example 10.3 of \cite{GLS} to get the general picture of how the cluster algebra $\widehat{\mathcal{A}}$ is constructed generally from a $2$-step parabolic subgroup, that is, a parabolic subgroup $P_K^-$ with $|J|=2$. Let $M:=n+(n-j_2)(j_2-1)$ and
\begin{align*}
A &:= \left\{n-1, n \right\} \\
&\cup \Big \{ a \mid a\in [n+1,M] \quad \textnormal{and} \quad i_a \geq j_2 -1  \Big\} \\
&\cup \left\{\begin{aligned}
&M+1, \quad M+j_1+1, \quad M+2j_1+1, \quad ..., \quad M+(j_2-j_1-1)j_1+1,\\
&M+(j_2-j_1-1)j_1+2, \quad M+(j_2-j_1-1)j_1+3, \quad ..., \quad M+(j_2-j_1)j_1
\end{aligned}
\right\}.
\end{align*}
It is not hard to see that $s(i_k)=\infty$ if and only if $k \in A$. Thus, by Definition \ref{Cell cluster}, these $k$'s are exactly the indexes of the frozen variables of the desired initial seed of $\mathcal{A}=\CC[N_K]$, while the other indexes are the ones of the mutable variables. More concretely, the frozen variables are
$${ D}_{\varpi_{i_k},w_{\leq k} \varpi_{i_k}}, \quad (k \in A),$$
while the mutable variables are $${D}_{\varpi_{i_l},w_{\leq l} \varpi_{i_l}}, \quad (l \notin A).$$
Also, by Definition \ref{Cell cluster}, the exchange matrix of this seed is
$$B_{jk} = \begin{cases}
1, & \text{if } j=p(k), \\
-1, & \text{if } j=s(k),\\
a_{i_j i_k}, & \text{if } j<k<s(j)<s(k),\\
-a_{i_j i_k}, & \text{if } k<j<s(k)<s(j),\\
0, & \text{otherwise;}
\end{cases}$$
Now, one can rewrite the indexes of the flag minors and find the use the Plucker relation to homogenize the lift of each exchange relation.
\end{example}

\begin{example}[Maximal parabolic subgroups of type $B$]
Let $G$ be a semisimple algebraic group of type $B_n$. Let $J=\{j\}$ and $K=\{1,...,n\} \setminus \{j\}.$ The description of the longest word $w_0$ here depends on the position of $j$. We generalize Example \ref{explicit type B} by taking $j=n$. Here $w_0$ is constructed by taking the longest word $u$ of $A_{n-1}$ indexed by $1,...,n-1$, completed by some subword generating $N_k$, for instance, $$v=s_ns_{n-1}...s_1s_ns_{n-1}...s_2s_ns_{n-1}...s_3s_ns_{n-1}s_n.$$
Therefore, in this case, we have $w_0=uv.$ Let
$$A=\left \{n, 2n-1, 3n-3, 4n-5,...,\dfrac{n(n+1)}{2}-1, \dfrac{n(n+1)}{2} \right \}.$$
Note that $s(i_k)=\infty$ if and only if $k \in A$. Thus, the frozen variables are those indexed by such $k$'s. Same as before, the frozen variables are
$${ D}_{\varpi_{i_k},w_{\leq k} \varpi_{i_k}}, \quad (k \in A),$$
while the mutable variables are $${D}_{\varpi_{i_l},w_{\leq l} \varpi_{i_l}}, \quad (l \notin A).$$
Again, Definition \ref{Cell cluster}, gives the recipe of the exchange matrix $B$.\\

Now, the exchange matrix $\widehat{B}$ of the lift of this seed to the cluster algebra $\widehat{\mathcal{A}} \subset \CC[G/P_K^-]$ is $B$ extended by the row indexed by $j=n$, whose entries are obtained from Theorem \ref{main}. The frozen variables are $\Delta_{\varpi_n,\varpi_n}$ together with
$$\widetilde{D}_{\varpi_{i_k},w_{\leq k} \varpi_{i_k}}=\dfrac{\Delta_{\varpi_{i_k},w_{\leq k} \varpi_{i_k}}\Delta_{\varpi_{i_1}, \varpi_{i_1}}^{d_k}}{\Delta_{\varpi_{i_k}, \varpi_{i_k}}},$$
where $k \in A$. On the other hand, the mutable ones are
$$\widetilde{ D}_{\varpi_{i_l},w_{\leq l} \varpi_{i_l}}=\dfrac{\Delta_{\varpi_{i_l},w_{\leq l} \varpi_{i_l}}\Delta_{\varpi_{i_1}, \varpi_{i_1}}^{d_l}}{\Delta_{\varpi_{i_l}, \varpi_{i_l}}},$$
where $l \notin A$.

\end{example}

\end{document}